\documentclass[english]{amsart}
\usepackage{babel}
\usepackage{dsfont}
\usepackage{amsmath,amsthm,amssymb,babel,amsfonts,graphics}
\usepackage{caption}
\usepackage{subcaption}
\usepackage{tikz}

\usepackage[pdftex,colorlinks]{hyperref}

\newcommand{\kom}[1]{}
\renewcommand{\kom}[1]{{\bf [#1]}}

\def\vint_#1{\mathchoice%
          {\mathop{\kern 0.2em\vrule width 0.6em height 0.69678ex depth -0.58065ex
                  \kern -0.8em \intop}\nolimits_{\kern -0.4em#1}}%
          {\mathop{\kern 0.1em\vrule width 0.5em height 0.69678ex depth -0.60387ex
                  \kern -0.6em \intop}\nolimits_{#1}}%
          {\mathop{\kern 0.1em\vrule width 0.5em height 0.69678ex
              depth -0.60387ex
                  \kern -0.6em \intop}\nolimits_{#1}}%
          {\mathop{\kern 0.1em\vrule width 0.5em height 0.69678ex depth -0.60387ex
                  \kern -0.6em \intop}\nolimits_{#1}}}
                  
                  \newcommand{\aveint}[2]{\mathchoice%
          {\mathop{\kern 0.2em\vrule width 0.6em height 0.69678ex depth -0.58065ex
                  \kern -0.8em \intop}\nolimits_{\kern -0.45em#1}^{#2}}%
          {\mathop{\kern 0.1em\vrule width 0.5em height 0.69678ex depth -0.60387ex
                  \kern -0.6em \intop}\nolimits_{#1}^{#2}}%
          {\mathop{\kern 0.1em\vrule width 0.5em height 0.69678ex depth -0.60387ex
                  \kern -0.6em \intop}\nolimits_{#1}^{#2}}%
          {\mathop{\kern 0.1em\vrule width 0.5em height 0.69678ex depth -0.60387ex
                  \kern -0.6em \intop}\nolimits_{#1}^{#2}}}

\usepackage{color}
\usepackage{epsfig}

\def\1{\raisebox{2pt}{\rm{$\chi$}}}

\newcommand{\dist}{\operatorname{dist}}

\newcommand{\kint}{\vint}

\newcommand{\eps}{\varepsilon}

\theoremstyle{plain}
\newtheorem{definition}{Definition}[section]
\newtheorem{proposition}[definition]{Proposition}
\newtheorem{theorem}[definition]{Theorem}
\newtheorem*{theorem*}{Main Theorem}

\newtheorem{lemma}[definition]{Lemma}

\theoremstyle{definition}

\theoremstyle{remark}

\numberwithin{equation}{section}

\begin{document}

\title[]{Tug-of-war games associated with boundary value problems involving derivatives}

\author{Jeongmin Han}
\address{Department of Mathematics, Soongsil University, 
06978 Seoul, Republic of Korea}
\email{jeongmin.han@ssu.ac.kr}


\keywords{Dynamic programming principle, stochastic game, $p$-Laplacian, viscosity solution}
\subjclass[2020]{91A05, 91A15, 35D40, 35B65}

\begin{abstract}
In this paper, we study a certain type of noisy tug-of-war game
which can be regarded as an interpretation of a certain type of
boundary value problem for the normalized $p$-Laplace equation, where $1<p<2$. 
More precisely, we will investigate the boundary regularity of the value function to the game
and its convergence to a viscosity solution of the model problem. 
\end{abstract}

\maketitle
\tableofcontents

\section{Introduction}
In this paper, we consider a certain type of stochastic game related to the following boundary value problem
\begin{align} \label{defcapf}
\left\{ \begin{array}{ll}
\Delta_{p}^{N}u=0 & \textrm{in $ \Omega$,}\\
\langle \mathbf{n}  , Du  \rangle + \gamma u = \gamma G & \textrm{on $ \partial \Omega$,}\\
\end{array} \right.
\end{align}
where $\Delta_p^{N}$ is the normalized $p$-Laplace operator given by $$\Delta_{p}^{N}u=\Delta u + (p-2)\frac{\langle D^2 u Du, Du \rangle}{|Du|^2},$$
$1<p<2$, $\gamma >0$ and  $\Omega\subset \mathbb{R}^n$ is a bounded domain.
To do this, we will first observe a mean value characterization linked to the above problem and 
construct a tug-of-war game related to the dynamic programming principle (DPP) based on the characterization
(for a heuristic description of the game setting, see Section \ref{subsec:mvc}).
In particular, we focus on the regularity and convergence results of value functions for the games associated with \eqref{defcapf}, which will be stated in Theorem \ref{bdryregext} and Theorem \ref{convrob12}.

This work can be regarded as an extension of \cite{han2024tug}.
In that paper, the author considered a game-theoretic interpretation of \eqref{defcapf} with $p>2$,
and the game is motivated by the stochastic processes in \cite{MR3011990,MR3450753,MR4684373}.
To cover the case $1<p<2$, we will present a different type of tug-of-war games:
We will use the `noisy tug-of-war game' in \cite{MR4257616} to construct the stochastic game for the model problem.
This construction enables us to use the Krylov-Safonov theory for certain types of DPPs in \cite{MR4496902,MR4565418,MR4739830}.

Our strategy to prove the main results is as follows.
For the regularity issue, we derive H\"{o}lder estimates for the interior case and the boundary case. 
The interior estimate can be shown by using the Krylov-Safonov type estimates in \cite{MR4496902,MR4565418,MR4739830}.
We need to go through several steps to obtain the boundary regularity, Theorem \ref{bdryregext} (cf. \cite{MR4684385,han2024tug}).
The key to the proof is a supermartingale argument employing the optional stopping theorem.
Constructing an appropriate supermartingale is crucial for this method  
and to this end, we present some geometric and stochastic observations relevant to our game beforehand.
Meanwhile, we will show the convergence of value functions by using the notion of viscosity solutions (cf. \cite{MR3011990}).
 
The game setting basically includes random noise when the tug-of-war occurs.
If $p>2$, we can construct the associated game with radially symmetric noise for the tug-of-war process.
This property makes us to deal with the terms caused by that noise easier than the other case.
When $1<p<2$, however, we can only consider the game with noise that is not rotationally invariant.
We need to introduce a suitable argument to derive our desired estimate in this case (see the proof of Theorem \ref{bdryregext}).

Tug-of-war and its noise-included version were first introduced in \cite{MR2451291,MR2449057}.
Since then, these stochastic games have been investigated as a game-theoretic interpretation of $p$-Laplace type operators.
General introductions to the connection between tug-of-war and the $p$-Laplacian can be found in \cite{MR4174394,MR4759243}.
The construction of stochastic games is based on mean value characterizations for solutions to the model equations.
We refer to \cite{MR2566554,MR3011990} for mean value characterizations for $p$-harmonic functions.
Especially, games associated with the normalized $p$-Laplace operators with $1<p<\infty$ were studied in \cite{MR2875296,MR3623556,MR4125101}.
In \cite{MR3182683}, probabilistic approaches are considered for the Neumann boundary value problem for fully nonlinear equations.
One can also find preceding studies \cite{MR4684373,MR4684385,han2024tug} which deal with stochastic processes
related to the boundary value problems involving derivatives for $p$-Laplace type equations.
We also mention Krylov-Safonov type regularity estimates for stochastic games and DPPs in \cite{MR4496902,MR4565418}.
Meanwhile, the boundary condition in \eqref{defcapf} is a special case of oblique derivative boundary conditions, which is given by 
$ \langle \mathbf{\beta}  , Du  \rangle + \gamma u = \gamma G $
with $|\langle \beta, \mathbf{n}\rangle | \ge \delta_{0} $ for some $\delta_0>0$.
We refer to \cite{MR3059278} for the general theory about oblique derivative boundary value problems.
Several works \cite{MR3780142,MR4223051} gave regularity results for fully nonlinear oblique derivative boundary value problems.
In \cite{MR4046185,MR4283880}, Calderon-Zygmund type theory was considered for these problems.

{\bf Acknowledgments} 
This research was supported by the Soongsil University Research Fund (New Professor Support Research) of 2025.
The author would like to thank M. Lewicka, M. Parviainen and E. Ruosteenoja for their useful discussions.

\section{Preliminaries}
\label{sec:prelim}

\subsection{Notations}
\begin{itemize}
\item For two $n$-dimensional vectors $a=(a_{1},\dots,a_{n})$ and $b=(b_{1},\dots,b_{n})$,
$$ \langle a,b \rangle = \sum_{i=1}^{n} a_{i}b_{i} .$$
\item For two $n\times n$ matrices $A=(a_{ij})$ and $B=(b_{ij})$,
$$ \langle A:B \rangle = \sum_{i,j=1}^{n} a_{ij}b_{ij}.$$
\item We write the $k$-dimensional ball with the radius $r$ as $B_{r}^{k}$, and $$B_{r,d}^{k}=B_{r}^{k}\cap \{y_{k}< d\}$$ for $0\le d\le r$.
\item $\Omega$ is a bounded $C^{1,1}$-domain in $\mathbb{R}^n$.
\item For each $x\in \Omega$, the projection of $x$ to $\partial\Omega$ will be denoted by $\pi_{\partial\Omega}x$.
\item For each $x\in \overline{\Omega}$ and given $\eps>0$, \begin{align*}d_{\eps}(x) =\min \bigg\{ 1, \frac{1}{\eps} \dist(x, \partial \Omega) \bigg\} \quad \textrm{and} \quad 
d'_{\eps}(x) =\min \bigg\{ \frac{1}{2}, \frac{1}{\eps} \dist(x, \partial \Omega) \bigg\}. \end{align*}
If there is no room for confusion, we abbreviate $d_{\eps}(x), d'_{\eps}(x)$ to $d_{x}, d'_{x}$.
\item For $r>0$ and given domain $\Omega$, we write $$\Gamma_{r}:=\{x\in \overline{\Omega}: \dist(x, \partial\Omega)\le r \}.$$
\item For a measurable set $L \subset \mathbb{R}^n$ and a measurable function $f$, we write
$$\kint_L f dx = \frac{1}{|L|}\int_L f dx.$$
\item We denote by $$E_{\eps}^{a}(z;\nu)=z+\bigg\{ y\in \mathbb{R}^{n}: \frac{\langle y, \nu\rangle^{2}}{a^{2}}+ |y-\langle y,\nu\rangle\nu|^{2}< \eps^{2} \bigg\} .$$
\end{itemize}

\subsection{Mean value characterization and game setting}
\label{subsec:mvc}
In this subsection, we consider a mean value characterization and the associated game related to the problem \eqref{defcapf} for $1<p<2$.
We can refer the readers to \cite{MR2875296,MR3623556,MR4125101}, which deal with tug-of-war games covering the same model equation. 
One can also find a study on noisy tug-of-war games in \cite{MR4257616},
which will be employed in our stochastic game setting.
This setting guarantees uniform ellipticity, and it enables us to employ the Krylov-Safonov type regularity results in \cite{MR4657420,MR4739830},
which will be introduced in the next section.
 
We basically consider mean value characterizations and corresponding stochastic games based on the noisy tug-of-war game in \cite{MR4257616}. 
Under the setting, at each point $x\in\Omega$, we will take strategies of two players on the $\eps$-sphere $\partial B_{\eps}(x)$ so that
the value function corresponds to the following DPP
\begin{align} \label{aveell}
u_{\eps}(x)=\frac{1}{2}\bigg( \sup_{z\in \partial B_{\eps}(x)} h_{u_{\eps}}^{A,a}(z;x,\eps)+  
 \inf_{z\in \partial B_{\eps}(x)} h_{u_{\eps}}^{A,a}(z;x,\eps) \bigg),
\end{align}
where
\begin{align}\label{defh}
  h_{u_{\eps}}^{A,a}(z;x,\eps): &=
 \kint_{B_{1}}u_{\eps}\bigg(z+ A \eps y + \frac{A (a-1)}{\eps}\langle y, z-x \rangle
 (z-x) \bigg) dy
\end{align}
with $ A, a >0$ satisfying
\begin{align}\label{aaqd}
\frac{n+2}{A^{2}}+a^{2}=p-1.
\end{align}
Then we can rewrite \eqref{defh} as
$$  h_{u_{\eps}}^{A,a}(z;x,\eps)= \kint_{E_{A\eps}^{a}(z;\frac{z-x}{|z-x|})}u_{\eps}(y)dy .$$
When $1<p<2$, we have to take $a<1$.
This means that the distribution of the noise when the tug-of-war occurs depends on the strategies.
Compared to the work \cite{han2024tug} for the case $p\ge 2$, it causes additional difficulties in deriving our desired regularity result 
because we cannot use a cancellation argument directly to estimate the terms related to the noise when the tug-of-war occurs.

Consider a domain $\Omega$ satisfying $B_{\eps}(x)\cap\Omega=x+B_{\eps,\eps d_{x}}^{n} $.
Then the following properties can be found in \cite[Lemma 2.1]{MR4684373} and \cite[Lemma 2.4]{MR4684373}:
\begin{align}\label{bdryintav} \kint_{B_{\eps}(x)\cap\Omega}(y-x)dy = -s_{\eps}(x)\mathbf{e}_{n},
\end{align} where 
\begin{align}\label{defs}s_{\eps}(x)= \frac{|B_{1}^{n-1}|}{(n+1)|B_{1,d_{\eps}(x)}^{n}|}\eps (1 -d_{\eps}(x)^{2} )^{\frac{n+1}{2}}
\end{align} 
and
\begin{align}\label{secorder}\kint_{B_{\eps}(x)\cap\Omega}(y-x)\otimes(y-x)dy=\frac{\eps^2}{n+2}I_n + O(\eps s_{\eps}(x)),
\end{align} 
where $I_n$ is the $n\times n$ identity matrix.

We present a mean value characterization for $p$-harmonic functions with respect to \eqref{aveell}
by employing \cite[Theorem 2.1]{MR4257616}.
\begin{lemma} \label{elltow}
Let $1<p<2 $ and $u \in C^{2}(\Omega)$ be a $p$-harmonic function with $Du(x)\neq 0$ for each $x\in \overline{\Omega}$.
Also, let $A$ and $a$ be fixed positive real numbers satisfying \eqref{aaqd}.
Then, we have
\begin{align*}
&\frac{1}{2}\bigg( \sup_{z\in \partial B_{\eps}(x)} h_{u}^{A,a}(z;x,\eps)+  
 \inf_{z\in \partial B_{\eps}(x)} h_{u}^{A,a}(z;x,\eps) \bigg)
  =u(x)+\frac{A^{2}\eps^{2}}{2(n+2)}\Delta_{p}^{N}u(x)+o(\eps^{2}).
\end{align*}
\end{lemma}

Fix $p_{0} \in (1,p)$ and set $A,a>0$ with 
\begin{align}\label{aaqdef}
\frac{n+2}{A^{2}}+a^{2}=p_{0}-1.
\end{align}  
We also recall the notation $$E_{\eps}^{a}(z;\nu)=z+\bigg\{ y\in \mathbb{R}^{n}: \frac{\langle y, \nu\rangle^{2}}{a^{2}}+ |y-\langle y,\nu\rangle\nu|^{2}< \eps^{2} \bigg\} .$$  
Now we define an operator $\mathcal{A}$ by
\begin{align} \label{defopa}\begin{split} \mathcal{A}u_{\eps}(x)&=
 \frac{\alpha}{2} \bigg\{\sup_{z\in \partial B_{Q\eps d'_{x}}(x)}\kint_{E_{QA\eps d'_{x}}^{a}(z;\nu_{0})} \hspace{-2em}u_{\eps}(y)dy+\inf_{z\in \partial B_{Q\eps d'_{x}}(x)} \kint_{E_{QA\eps d'_{x}}^{a}(z;\nu_{0})}\hspace{-2em} u_{\eps}(y)dy \bigg\}\\ & \quad + (1- \alpha)\kint_{B_{\eps}(x)\cap \Omega}u_{\eps}(y)dy,
 \end{split}
\end{align} 
where $\alpha=\frac{4(2-p)}{4(2-p)+(p-p_{0})(QA)^{2}} \in(0,1)$,
 $ \nu_{0}=\frac{z-x}{|z-x|}$ and $Q>0$ with $Q(1+A) \le 1$.
Then, \eqref{defopa} can be rewritten as
\begin{align}\label{defopaint}\begin{split} \mathcal{A}u_{\eps}(x)=&\frac{\alpha}{2} \bigg\{\sup_{z\in \partial B_{Q\eps /2}(x)}\kint_{E_{QA\eps/2}^{a}(z;\nu_{0})}u_{\eps}(y)dy+\inf_{z\in \partial B_{Q\eps/2}(x)} \kint_{E_{QA\eps/2}^{a}(z;\nu_{0})} u_{\eps}(y)dy \bigg\}\\ & \quad + (1- \alpha)\kint_{B_{\eps}(x)}u_{\eps}(y)dy\end{split}
\end{align}
 when $\dist(x,\partial \Omega) \ge \frac{\eps}{2} $.

For the boundary value problem \eqref{defcapf}, we can derive the following mean value characterization.
\begin{proposition}
Let $u \in C^{2}(\overline{\Omega})$ be a function solving the problem \eqref{defcapf} with $G \in C^{1}(\overline{\Gamma}_{r_{0}})$ for some $r_{0}>0$. Also, let $A$ and $a$ be fixed positive real numbers satisfying \eqref{aaqdef} and
$B_{\eps d_{\eps}(x) Q(1+A)}(x)\subset \Omega$.
Assume that $Du(x)\neq 0$ for each $x\in \overline{\Omega}$.
Then, we have 
\begin{align*} 
u(x) =\big( 1-(1-\alpha)\gamma_{0} s_{\eps}(x) \big) \mathcal{A}u(x) +
(1-\alpha)\gamma_{0} s_{\eps}(x) G(x)
+O(\eps s_{\eps}(x)  )+o(\eps^{2}).
\end{align*}
\end{proposition}
\begin{proof}
We first observe that
\begin{align*} \kint_{B_{\eps}(x)\cap \Omega }u(y)dy&=u(x)+\bigg\langle Du(x),  \kint_{B_{\eps}(x)\cap \Omega }(y-x)dy \bigg\rangle
\\ &\quad + \frac{1}{2}\bigg\langle D^{2}u(x):\kint_{B_{\eps}(x)\cap \Omega }(y-x)\otimes(y-x)dy\bigg\rangle + o(\eps^{2})
\end{align*}
by Taylor expansion.
We also see that 
\begin{align*}
\bigg\langle Du(x), \kint_{B_{\eps}(x)\cap \Omega}(y-x) dy \bigg\rangle
 = \gamma_{0} s_{\eps}(x)(u(\pi_{\partial \Omega}x))-G(\pi_{\partial \Omega}x)))
+ O(\eps s_{\eps}(x)  ),
\end{align*}
($\pi_{\partial \Omega}$ is the projection of $x$ on $\partial \Omega$) and
\begin{align*}
\bigg\langle D^{2}u(x):\kint_{B_{\eps}(x)\cap \Omega }(y-x)\otimes(y-x)dy\bigg\rangle
= \frac{\Delta u(x)}{n+2}\eps^{2}+O(\eps s_{\eps}(x))
\end{align*}
in \cite[Lemma 2.3]{MR4684373} and  \cite[Lemma 2.4]{MR4684373}, respectively.

Recall \eqref{defopaint}, that is,
\begin{align*}
 \mathcal{A}u(x)=&\frac{\alpha}{2} \bigg\{\sup_{z\in \partial B_{Q\eps /2}(x)}\kint_{E_{QA\eps/2}^{a}(z;\nu_{0})}u(y)dy+\inf_{z\in \partial B_{Q\eps/2}(x)} \kint_{E_{QA\eps/2}^{a}(z;\nu_{0})} u(y)dy \bigg\}\\ & \quad + (1- \alpha)\kint_{B_{\eps}(x)}u(y)dy.
\end{align*}
Then, by Lemma \ref{elltow}, we have
\begin{align} \label{add5}
\mathcal{A}u(x)=u(x)+\frac{\alpha A^{2}Q^{2}}{8(n+2)}\Delta_{p_0}^{N}u(x)\eps^{2}+  (1-\alpha)\frac{\Delta u(x)}{2(n+2)}\eps^{2}+o(\eps^{2})
=u(x)+o(\eps^{2})
\end{align}
since $\alpha=\frac{4(2-p)}{4(2-p)+(p-p_{0})(QA)^{2}} $ when $\dist(x,\partial\Omega)\ge \eps$.

Next, we assume that $\frac{\eps}{2}\le \dist(x,\partial \Omega)<\eps$.
Since $d'_{x}=\frac{1}{2}$, we obtain
\begin{align*}\mathcal{A}u(x)  &
=  u(x)+\frac{\alpha A^{2}Q^{2}}{8(n+2)}\Delta_{p_0}^{N}u(x)\eps^{2}+\frac{1-\alpha}{2(n+2)}\Delta u(x)\eps^{2}  
\\ &\quad+
(1-\alpha) \gamma_{0} s_{\eps}(x)(u(\pi_{\partial \Omega}x)-G(\pi_{\partial \Omega}x))+ O(\eps s_{\eps}(x)  )+o(\eps^{2})
\\ &=  u(x)+ (1-\alpha) \gamma_{0} s_{\eps}(x) (u(x)-G(x))+ O(\eps s_{\eps}(x)  )+o(\eps^{2}).
\end{align*}
Note that the last equation can be derived by using $u\in C^{2}(\overline{\Omega})$ and $G\in C^{1}(\Gamma_{r_{0}})$.
Therefore, we have 
\begin{align*}u(x)= & \frac{\mathcal{A}u(x)}{1+(1-\alpha)\gamma_{0} s_{\eps}(x) }
+\frac{(1-\alpha)\gamma_{0} s_{\eps}(x)}{1+(1-\alpha)\gamma_{0} s_{\eps}(x)}G(x)+ O(\eps s_{\eps}(x)  )+o(\eps^{2}).
\end{align*}
Now we observe
$$ \frac{1}{1+(1-\alpha)\gamma_{0} s_{\eps}(x) }= 1- (1-\alpha)\gamma_{0} s_{\eps}(x) + O(\eps s_{\eps}(x)),$$
and thus,
\begin{align}\label{add4}\begin{split}u(x)= & \big(1-(1-\alpha)\gamma_{0} s_{\eps}(x) \big)\mathcal{A}u(x)+ (1-\alpha)\gamma_{0} s_{\eps}(x)G(x)+ O(\eps s_{\eps}(x)  )+o(\eps^{2}).\end{split}
\end{align}  

Finally, we consider the case $\dist(x,\partial \Omega)< \frac{\eps}{2}$. 
In this case, we have
\begin{align*} \mathcal{A}u(x)&=  u(x)+\bigg(\frac{\alpha A^{2}Q^{2}(d'_x)^2}{2(n+2)}\Delta_{p_0}^{N}u(x)\eps^{2} +\frac{1-\alpha}{2(n+2)}\Delta u(x)\bigg)\eps^{2}
\\ &\quad+(1-\alpha) \gamma_{0} s_{\eps}(x)(u(\pi_{\partial \Omega}x)-G(\pi_{\partial \Omega}x)) + O(\eps s_{\eps}(x)  )
 \\ & = u(x)+ (1-\alpha) \gamma_{0} s_{\eps}(x) (u(x)-G(x)) + O(\eps s_{\eps}(x)  )
\end{align*}
because $$\bigg(\frac{\alpha A^{2}Q^{2}(d'_x)^2}{2(n+2)}\Delta_{p_0}^{N}u(x)\eps^{2}+\frac{1-\alpha}{2(n+2)}\Delta u(x)\bigg)\eps^{2}
=O(\eps^{2})=O(\eps s_{\eps}(x))$$
by using the assumption $u\in C^{2}(\overline{\Omega})$.
This yields
\begin{align}\label{add6}
u(x)= & \big(1-(1-\alpha)\gamma_{0} s_{\eps}(x) \big)\mathcal{A}u(x)+ (1-\alpha)\gamma_{0} s_{\eps}(x)G(x)+ O(\eps s_{\eps}(x)  ).
\end{align}
The proof is completed by combining \eqref{add5}, \eqref{add4} with \eqref{add6}.
\end{proof}
The DPP  
\begin{align}\label{dppext} 
u_{\eps}(x) =\big( 1-(1-\alpha)\gamma_{0} s_{\eps}(x) \big) \mathcal{A}u_{\eps}(x) +
(1-\alpha)\gamma_{0} s_{\eps}(x) G(x)
\end{align}
naturally arises from the above proposition. 
We give a heuristic explanation of our game associated with this DPP here.
Assume that the game begins at $x=:x_0\in \Omega$. 
With probability $\gamma s_{\eps}(x_0)$, the game ends right at this point.
Otherwise, each player chooses a unit vector as his or her strategy and they play a tug-of-war game with noise.
In that case, with probability $\alpha$, there is a fair coin toss 
and the token is randomly chosen in $E_{QA\eps d'_{x_0}}(x_0+Q\eps d'_{x_0};\nu) $, where $\nu$ is the coin toss winner's strategy.
Meanwhile, with probability $1-\alpha$, the token is randomly chosen in $B_{\eps}(x_0)$.
Now we write $x_1$ to be the new location of the token and repeat the above procedures.
Similarly, we can set $x_2,x_3,\dots$ until the game ends.
Let the point where the game ends be denoted by $x_{\tau}$.
The value function of this game for each player is defined by
$$ u_{I}^{\eps}(x_{0})=\sup_{S_{I}}\inf_{S_{II}}\mathbb{E}_{S_{I},S_{II}}^{x_{0}} [G(x_{\tau})]
\quad \textrm{and} \quad u_{II}^{\eps}(x_0)=\inf_{S_{II}}\sup_{S_{I}}\mathbb{E}_{S_{I},S_{II}}^{x_{0}} [G(x_{\tau})], $$
where $S_{I}$ and $S_{II}$ are the strategies for Player I and II, respectively.
If $ u_{I}^{\eps}\equiv  u_{II}^{\eps}$, we call the function the value of the game.

We can describe the game more rigorously as follows.
For each $k=1,2,\dots$, $\xi_{k}\in (0,1)$ is randomly chosen with respect to the uniform distribution, 
and $c_k \in \{0,1\}$ is given by
\begin{align*} c_{k} =
\left\{ \begin{array}{ll}
0 & \textrm{if $ \xi_{k-1} \le 1-\gamma s_{\eps}(x)  $,}\\
1 & \textrm{if $ \xi_{k-1} > 1-\gamma s_{\eps}(x)$.}\\
\end{array} \right.
\end{align*}
Next, we define a strategy as a Borel-measurable function designating the next position of the token.
For $j \in \{ I, II\}$, we set the strategy
$S_{j}^{k} \in \partial B_{ Q/2}(X_k) $ of each Player
for each $k$. 
Let 
\begin{align*} \Upsilon_{1}=(B_1)^{\mathbb{N}}\times(0,1)=\{ (w,b)
: w=\{w^{j}\}_{j=1}^{\infty}, w^{j}\in B_{1}^{n} 
\textrm{ for all } j\in \mathbb{N} , \ b\in (0,1)  \},
\end{align*}
$\mathcal{F}_1$ be the Borel $\sigma$-algebra, and $\mathbb{P}^{S_{I},S_{II}}  $  be the probability measure with respect to the strategies $S_{I}$ and $S_{II}$ defined as 
$$\mathbb{P}^{S_{I},S_{II}}(T) =\frac{\alpha}{2}\big( \kint_{E_{QA/2}^{a}(S_{I};\nu_{S_{I}})} dh + \kint_{E_{QA/2}^{a}(S_{II};\nu_{S_{II}})} dh \big) + (1-\alpha)\kint_{B_1 \cap T} dh$$
for every $T\in \mathcal{F}_1$, where $\nu_{S_{I}}=\frac{S_{I}}{|S_{I}|}$ and $\nu_{S_{II}}=\frac{S_{II}}{|S_{II}|}$.

We define $(\Upsilon,\mathcal{F},\mathbb{P}^{S_{I},S_{II}}) $ by
the countable product of the probability space $(\Upsilon_{1},\mathcal{F}_{1},\mathbb{P}_{1}^{S_{I},S_{II}}) $.
Furthermore, for each $k \in \mathbb{N}$, we also define $(\Upsilon_{k},\mathcal{F}_{k},\mathbb{P}_{k}^{S_{I},S_{II}}) $ by the product on $k$-copies of $(\Upsilon_{1},\mathcal{F}_{1},\mathbb{P}_{1}^{S_{I},S_{II}})$.
We note that $\{ \mathcal{F}_{k} \}_{n=0}^{\infty}$ with $\mathcal{F}_{0}=\{ \varnothing, \Upsilon \} $ is a filtration of $\mathcal{F}$, and $\mathcal{F}_{k} $ is identified with the sub-$\sigma$-algebra of $\mathcal{F}$
consisting of sets $A\times \prod_{i=k+1}^{\infty}\Upsilon_{1}$ for every $A \in \mathcal{F}_{k} $.
On the other hand, we set the sequence of measurable functions $\{ l_{i}^{\eps} : \Upsilon \times \overline{\Omega} \to \mathbb{N}\cup \{+\infty \} \}_{i=1}^{\infty}$ by
$$l_{i}^{\eps}(\omega, x)=\min \{ l \ge 1 : x+\eps w_{i}^{l} \in B_{\eps}(x)\cap \Omega
\} \qquad \textrm{for every } \omega \in \Upsilon, \ x \in \overline{\Omega}.$$
Then we can consider the sequence of vector-valued random variables $\{ w_{i}^{\eps, x}\}_{i=1}$ by
$$ w_{i}^{\eps, x}(\omega)=w_{i}^{l_{i}^{\eps} (\omega, x)} \qquad  \textrm{for $\mathbb{P}$-a.e. } \omega \in \Upsilon.  $$ since $l_{i}^{\eps} $ is $\mathbb{P}$-a.s. finite.

Now we consider rotations $\{\mathcal{P}_{\nu}\}$ for each $\nu\in S^{n-1}$ such that $\mathcal{P}_{\nu}\mathbf{e}_{1}=\nu$
with \begin{align} \label{conrot}
    \mathcal{P}_{\nu_{1}}\zeta-\mathcal{P}_{\nu_{2}}\zeta |\le C|\nu_{1}-\nu_{2}|
\end{align}| for any $\nu_{1},\nu_{2}\in S^{n-1}$, $\zeta \in B_{1}$ and some universal $C>0$ (this can be shown as in \cite[Lemma A.1]{MR4125101})
and a linear transformation $\mathcal{D}:\mathbb{R}^{n}\to \mathbb{R}^{n}$ to be $$\mathcal{D}\mathbf{e}_{1}=a\mathbf{e}_{1} \quad \textrm{and} \quad \mathcal{D}\mathbf{e}_{j}=\mathbf{e}_{j} \quad \textrm{for} \quad j=2,\dots,n.$$ 
Now we define the sequence of vector-valued random variables $\{ X_{k}^{\eps, x_{0}}\}_{k=0}$ by
$X_{0}^{\eps, x_{0}} \equiv x_{0}$ and 
\begin{equation}X_{k}^{\eps, x_{0}} =
\left\{ \begin{array}{llll}
X_{k-1}^{\eps, x_{0}}+Q\eps d'_{X_{k-1}}S_{I}^{k-1}+QA\eps d'_{X_{k-1}}\mathcal{P}_{S_{I}^{k-1}}\mathcal{D}w_{1}^{X_{k-1}^{\eps, x_{0}}} & \\
& \hspace{-17em}\textrm{if $0< \xi_{k-1} \le \frac{\alpha}{2}  (1-\gamma s_{\eps}(X_{k-1}))$,}\\
X_{k-1}^{\eps, x_{0}}+Q\eps d'_{X_{k-1}}S_{II}^{k-1}+QA\eps d'_{X_{k-1}}\mathcal{P}_{S_{II}^{k-1}}\mathcal{D}w_{1}^{X_{k-1}^{\eps, x_{0}}}  &  \\
& \hspace{-17em}\textrm{if $  (1-\frac{\alpha}{2}) (1-\gamma s_{\eps}(X_{k-1})) \le \xi_{k-1} <  1-\gamma s_{\eps}(X_{k-1})$,}\\
X_{k-1}^{\eps, x_{0}}+\eps w_{k}^{X_{k-1}^{\eps, x_{0}}} & \\
& \hspace{-17em} \textrm{if $  \frac{\alpha}{2} (1-\gamma s_{\eps}(X_{k-1})) < \xi_{k-1} < (1-\frac{\alpha}{2}) (1-\gamma s_{\eps}(X_{k-1})) $,}\\
X_{k-1}^{\eps, x_{0}} & \hspace{-17em}\textrm{if $1-\gamma s_{\eps}(X_{k-1})< \xi_{k-1} < 1$}\\
\end{array} \right.
\end{equation}
for each $k=1,2,\dots$ and $\gamma=\gamma_{0}(1-\alpha)$.

We can verify that the function $u_{\eps}$ indeed coincides with the value function. 
This can be done through a similar process as in \cite[Section 3]{han2024tug}.

\begin{figure}
\begin{tikzpicture}
\centering
\draw[dotted, rotate=0] (0, 0.425) ellipse (0.425cm and 0.425cm);
\draw[rotate=0] (0, 0.85) ellipse (0.15cm and 0.05cm);
\draw[rotate=45] (0.3005, 0.7255) ellipse (0.15cm and 0.05cm);
\draw[rotate=135] (0.3005, 0.1245) ellipse (0.15cm and 0.05cm);
\draw[rotate=-135] (-0.3005, 0.1245) ellipse (0.15cm and 0.05cm);
\draw[rotate=-45] (-0.3005, 0.7255) ellipse (0.15cm and 0.05cm);
\draw[rotate=0] (0, 0) ellipse (0.15cm and 0.05cm);
\draw[rotate=0] (-0.425, 0.425) ellipse (0.05cm and 0.15cm);
\draw[rotate=0] (0.425, 0.425) ellipse (0.05cm and 0.15cm);
\draw[rotate=0] (0, 0.425) ellipse (1.5cm and 1.5cm);
\draw (-3.0, 0.9) -- (3.0, 0.9) node[below]{\footnotesize $\Omega$};
\filldraw[black]  (0,0.425) circle (1pt)  node [anchor=east] {\footnotesize $x$} ;
\draw[color=blue] (0,0) -- (0, 0.425) node at (-0.2,-0.5) {\footnotesize $Q\eps d'_{x}$} ;
\draw[color=blue] (-0.32, -0.25) -- (0, 0.2);
\draw[color=red] (0,0.425) -- (1.5, 0.425) node at (1,0.28) {\footnotesize $\eps$} ;
\draw[color=blue] (-0.32, -0.25) -- (0, 0.2);
\fill[gray!20,nearly transparent] (-3,0.9) -- (-3,-1.8) -- (3,-1.8) -- (3,0.9) -- cycle;
\end{tikzpicture}
\caption{Near the boundary case. }
\label{fig1}
\end{figure}
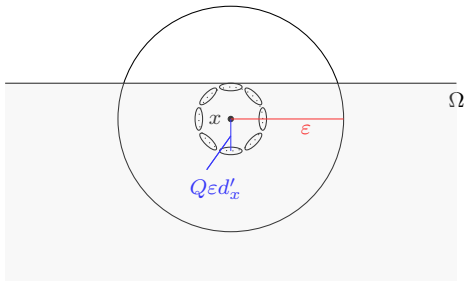

\section{Main results}
\label{sec:1p2}

In this section, we verify some properties of value functions for the game we constructed in Section \ref{sec:prelim}.
Precisely, we will present regularity (Theorem \ref{bdryregext}) and convergence (Theorem \ref{convrob12}) results for functions satisfying \eqref{dppext}.
For the regularity estimate, there have been several results for value functions of tug-of-war games by using coupling arguments, for example, \cite{MR3011990,MR3623556,MR4125101}.
In \cite{MR4684385,han2024tug}, the authors also considered regularity for the stochastic process related to boundary value problems involving derivatives.
Meanwhile, we will show the convergence of value functions when the step size $\eps$ goes to $0$.
We refer the reader to \cite{MR2875296,MR3011990,MR3623556}, for example. The author's previous result \cite[Section 5]{han2024tug} also discussed this issue for \eqref{defcapf} when $2<p<\infty$.

\subsection{Regularity} The regularity issue of our value function is important in its own right and also plays a crucial role in investigating the connection to the model problem, namely the convergence. In this subsection, we study the regularity of the value function in both the interior and boundary cases.
\subsubsection{Interior regularity}
First, we focus on the interior regularity of our value function. 
We will employ Krylov-Safonov type estimates for extremal operators in \cite{MR4496902,MR4565418} to derive the desired result, Theorem \ref{bdryregext}.  
It requires some preparation to do this before proving this.
We denote by $\mathcal{M}(B_{\Lambda})$ the set of symmetric unit Radon measures with
support in $B_{\Lambda}$.
We also set $\mu : \mathbb{R}^{n} \to \mathcal{M}(B_{\Lambda}) $ such that
$$x \mapsto \int_{B_{\Lambda}} u(x+h)d\mu_{x} (h),$$
which defines a Borel measurable function for every Borel measurable $u:  \mathbb{R}^{n} \to \mathbb{R}$.
The following notion of extremal operators can be found in \cite{MR4739830} (cf. \cite{MR4657420}).

\begin{definition}\label{exopdef}
Let $u: \mathbb{R}^{n} \to \mathbb{R}$ be a Borel measurable bounded function.
We define the extremal Pucci-type operators 
\begin{align*}
\mathcal{L}_{\eps}^{+}u(x)&
:=\frac{1}{2\eps^{2}} \bigg( \alpha \sup_{\mu \in \mathcal{M}(B_{\Lambda})} \int_{B_{\Lambda}}\delta u (x, \eps h)d\mu(h)+(1-\alpha)\kint_{B_{1}}\delta u (x, \eps h)dh\bigg)
\\ & = \frac{1}{2\eps^{2}} \bigg( \alpha \sup_{h \in B_{\Lambda}} \delta u (x, \eps h)+(1-\alpha)\kint_{B_{1}}\delta u (x, \eps h)dh\bigg)
\end{align*}
and
\begin{align*}
\mathcal{L}_{\eps}^{-}u(x)&
:=\frac{1}{2\eps^{2}} \bigg( \alpha \inf_{\mu \in \mathcal{M}(B_{\Lambda})} \int_{B_{\Lambda}}\delta u (x, \eps h)d\mu(h)+(1-\alpha)\kint_{B_{1}}\delta u (x, \eps h)dh\bigg)
\\ & = \frac{1}{2\eps^{2}} \bigg( \alpha \inf_{h \in B_{\Lambda}} \delta u (x, \eps h)+(1-\alpha)\kint_{B_{1}}\delta u (x, \eps h)dh\bigg),
\end{align*}
where $\delta u(x)=u(x+\eps h)+u(x-\eps h)-2u(x) $ for every $h \in B_{\Lambda}$.
\end{definition}
The following Krylov-Safonov type regularity result for such extremal operators
can be found in \cite{MR4565418,MR4496902}.
\begin{lemma}\label{hestpeo}
Let $\mathcal{L}_{\eps}^{+}$ and $ \mathcal{L}_{\eps}^{-}$ be the extremal operators as in Defintion \ref{exopdef} and $\rho\ge0$.
Then, there exists $\eps_{0}>0$ such that if $u$ satisfies $ \mathcal{L}_{\eps}^{+}u\ge -\rho$
and  $ \mathcal{L}_{\eps}^{-}u\le \rho$ in $B_{R}$ where $\eps <\eps_{0}R$,
there exist $C>0,\sigma \in (0,1)$ such that
\begin{align*}
|u(x)-u(z)|\le \frac{C}{R^{\sigma}}\big( ||u||_{L^{\infty}(B_{R})}+R^{2}\rho \big)
(|x-z|^{\sigma}+\eps^{\sigma})
\end{align*} 
for every $x,z\in B_{R/2}$.
\end{lemma}

Recall the function $u_{\eps}$ the solution to the DPP
\begin{align*}\mathcal{A}u_{\eps}(x)=&\frac{\alpha}{2} \bigg\{\sup_{z\in \partial B_{Q\eps /2}(x)}\kint_{E_{QA\eps/2}^{a}(z;\nu_{0})}u_{\eps}(y)dy+\inf_{z\in \partial B_{Q\eps/2}(x)} \kint_{E_{QA\eps/2}^{a}(z;\nu_{0})} u_{\eps}(y)dy \bigg\}\\ & \quad + (1- \alpha)\kint_{B_{\eps}(x)}u_{\eps}(y)dy.
\end{align*}
We consider two ellipsoids $$E_{1}^{a}(0;\nu)=\bigg\{ y\in \mathbb{R}^{n}: \frac{\langle y, \nu\rangle^{2}}{a^{2}}+ |y-\langle y,\nu\rangle\nu|^{2}< 1 \bigg\}$$ and 
$$E_{1}^{a}(0;-\nu)=\bigg\{ y\in \mathbb{R}^{n}: \frac{\langle y, \nu\rangle^{2}}{a^{2}}+ |y+\langle y,\nu\rangle\nu|^{2}< 1 \bigg\}.$$
Then, one can check that the function $u_{\eps}$ satisfies the assumption of Lemma \ref{hestpeo} by introducing a measure 
$$\mu_{\nu}(L)=\frac{1}{2}\bigg( \frac{|E_{1}^{a}(\nu;0)\cap L|}{|E_{1}^{a}(\nu;0)|}+\frac{|E_{1}^{a}(-\nu;0)\cap L|}{|E_{1}^{a}(-\nu;0)|}\bigg) $$ for each measurable set $L$ 
and an operator 
\begin{align*}
\mathcal{I}_{\eps}^{\nu}u_{\eps}(x)=\alpha\kint_{E_{Q\eps/2}^{a}(x+\frac{QA\eps}{2} \nu;\nu)}u_{\eps}(y)dy+ (1- \alpha)\kint_{B_{\eps}(x)}u_{\eps}(y)dy
\end{align*}
for each unit vector $\nu$.
We state this as a proposition.
\begin{proposition}
Let $1<p<2$ and $u_{\eps}$ be a bounded Borel measurable function satisfying \eqref{dppext}.
Then we have $\mathcal{L}_{\eps}^{+}u_{\eps}\ge 0$ and $\mathcal{L}_{\eps}^{-}u_{\eps}\le 0$
for some $0<\alpha<1$ depending on $n,p$, where $ \mathcal{L}_{\eps}^{+}$ and $\mathcal{L}_{\eps}^{-}$ are the extremal operators as in Definition \ref{exopdef}. 
\end{proposition}

The following lemma follows from Lemma \ref{hestpeo}.
\begin{lemma}\label{intlipext}Let $\Omega$ be a bounded domain and $B_{2r}(x_{0})\subset \Omega$ for some $R>0$.
Then, for the function satisfying \eqref{aveell}, it holds 
\begin{align} \label{intestext} |u_{\eps}(x)-u_{\eps}(z)| \le C||u||_{L^{\infty}(B_{2R})}\bigg( \frac{|x-z|^{\sigma}}{R^{\sigma}}+
\frac{\eps^{\sigma}}{R^{\sigma}} \bigg) 
\end{align} for each $x,z\in B_{R}(x_{0})$. Here $C>0$ and $\sigma\in (0,1)$ are independent of $\eps$.
\end{lemma}

\subsubsection{Boundary regularity}
Now we are concerned with the boundary regularity.
We assume that $\Omega$ satisfies the $C^{1,1}$-boundary regularity condition.
Then we observe that $\Omega$ also satisfies the interior ball condition with the radius $\rho$. This means, for each $x \in \partial\Omega$,
there exists a ball $B_{\rho}(z) \subset \Omega$ for some $z \in \Omega$ and $\rho >0$
with $x \in \partial B_{\rho}(z) $. 
In that case, for $0<r<\rho$, we observe that for every $y \in \Omega_{r}$, its projection $\pi_{\partial \Omega}(y)$ to $\partial \Omega$  is uniquely defined.
Thus one can find a point $z_{0} \in \Omega$ such that $ B_{\rho}(z_{0}) \subset \Omega$ and $\pi_{\partial \Omega}(y)\in \partial B_{\rho}(z_{0})  $,
and hence, it is possible to define a function $Z^{\rho}: I_{r} \to \Omega$ with
$Z^{\rho}(y)=z_{0} $.
We also set $\overline{\tau}^{\eps,\rho,h,x_{0}}, S_{k}^{\eps , x_{0}}$ to be
\begin{align}\label{deftaus}
    \overline{\tau}^{\eps,\rho,h,x_{0}}=\min\{ k\ge 0 : |Z^{\rho}(X_{k})-X_{k}|<\rho- h \} \quad \textrm{and} \quad S_{k}^{\eps , x_{0}}= \sum_{j=1}^{k}s_{\eps}(X_{k}) . 
\end{align}

We go through several steps to derive boundary regularity estimates for the value functions.
We will use the optional stopping theorem to obtain our main result. 
To do this, we need to construct a proper submartingale by considering alternative stochastic games as in preceding results such as \cite{MR3011990,han2024tug}.
The following geometric observation is necessary to get an estimate of stopping times for alternative games (cf. \cite[Lemma 4.5]{han2024tug}).
  
\begin{lemma}  \label{mcs1infext}
Let $\Omega$ be a domain satisfying the $C^{1,1}$-boundary regularity condition with the radius $\rho>0$.
Fix $r \in (0,\frac{\rho}{2})$ and $x_{0}\in I_{r}$.
Then there exists a constant $C_{0}>0$ depending on $n, q, \rho$ and $\alpha$ such that for any small
$\eps>0$ and a unit vector $\nu=\frac{x_{0}-y_{0}}{|x_{0}-y_{0}|}$,
\begin{align} \label{lpmo1infext} \begin{split}
\frac{\alpha}{2}\kint_{E_{QA\eps d'_{x_{0}}}^{a}(x_{0}+Q\eps d'_{x_{0}}\nu;\nu)}\hspace{-5.5em}|z-y_{0}|^{-q}& dz + \frac{\alpha}{2}\kint_{E_{ QA\eps d'_{x_{0}}}^{a}(x_{0}-Q\eps d'_{x_{0}}\nu;\nu)}\hspace{-5.5em}|z-y_{0}|^{-q}dz + (1-\alpha)
\kint_{B_{\eps}(x_{0})\cap \Omega}\hspace{-1.5em}|z-y_{0}|^{-q}dz \\ & \ge |x_{0}-y_{0}|^{-q}+C_{0}(s_{\eps}(x_{0})+\eps^{2}), \end{split}
\end{align}
where $q> \frac{n-1}{a}-1 $ and $y_{0}=Z^{\rho}(x_{0}) $. 
\end{lemma}
\begin{proof}
One can check that
\begin{align}\label{fote2}
&\kint_{B_{\eps}(x_{0})\cap \Omega}|z-y_{0}|^{-q}dz
\ge |x_{0}-y_{0}|^{-q}+ C_{1}(s_{\eps}(x_{0})+\eps^{2})
\end{align}  
(see the proof of \cite[Lemma 4.5]{han2024tug}). Hence, we only need to get an estimate for
\begin{align*} &\frac{1}{2}\kint_{E_{QA\eps d'_{x_{0}}}^{a}(x_{0}+Q\eps d'_{x_{0}}\nu;\nu)} \hspace{-5.5em}|z-y_{0}|^{-q}  dz + \frac{1}{2}\kint_{E_{QA\eps d'_{x_{0}}}^{a}(x_{0}-Q\eps d'_{x_{0}}\nu;\nu)}\hspace{-5.5em} |z-y_{0}|^{-q}dz.
\end{align*}

Recalling the definition, we see that
\begin{align*}
&E_{QA\eps d'_{x_{0}}}^{a}(x_{0}+Q\eps d'_{x_{0}}\nu;\nu)\\&=x_{0}+Q\eps d'_{x_{0}}\nu+\bigg\{ y\in \mathbb{R}^{n}: \frac{\langle y, \nu\rangle^{2}}{a^{2}}+ |y-\langle y,\nu\rangle\nu|^{2}< (QA\eps d'_{x_{0}})^{2} \bigg\}     
\end{align*}and 
\begin{align*}
&E_{QA\eps d'_{x_{0}}}^{a}(x_{0}-Q\eps d'_{x_{0}}\nu;\nu)\\&=x_{0}-Q\eps d'_{x_{0}}\nu+\bigg\{ y\in \mathbb{R}^{n}: \frac{\langle y, \nu\rangle^{2}}{a^{2}}+ |y-\langle y,\nu\rangle\nu|^{2}< (QA\eps d'_{x_{0}})^{2} \bigg\}     
\end{align*}
Set $\phi(z)=|z-y_{0}|^{-q}$.
We observe that for $\tilde{x} =x_{0}+Q\eps d'_{x_{0}}\nu,$ we have
$$  \phi(\tilde{x})=  \big(|x_{0}-y_{0}|+Q\eps d'_{x_{0}} \big)^{-q},$$
$$ D\phi(\tilde{x})= -q\big(|x_{0}-y_{0}|+Q\eps d'_{x_{0}} \big)^{-q-2}(x_{0}-y_{0}+Q\eps d'_{x_{0}}\nu)$$ and
\begin{align*}& D^{2}\phi(\tilde{x})
\\ & =q\big(|x_{0}-y_{0}|+Q\eps d'_{x_{0}} \big)^{-q-4}\\ & \quad \times \big( (q+2)(x_{0}+Q\eps d'_{x_{0}}\nu-y_{0})\otimes
(x_{0}+Q\eps d'_{x_{0}}\nu-y_{0})-(|x_{0}-y_{0}|+Q\eps d'_{x_{0}})^{2}I_{n} \big) 
.
\end{align*}
We also get
\begin{align*}
&\kint_{E_{QA\eps d'_{x_{0}}}^{a}(\tilde{x};\nu)} |z-y_{0}|^{-q}  dz
\\ & =
\big(|x_{0}-y_{0}|+Q\eps d'_{x_{0}} \big)^{-q}+\kint_{E_{QA\eps d'_{x_{0}}}^{a}(\tilde{x};\nu)}\frac{1}{2}\langle D^{2}\phi (\tilde{x})(z-\tilde{x}),z-\tilde{x} \rangle dz + o(\eps^{2})
\end{align*}
and
\begin{align*}
&\kint_{E_{QA\eps d'_{x_{0}}}^{a}(\tilde{x};\nu)}\langle D^{2}\phi (\tilde{x})(z-\tilde{x}),z-\tilde{x} \rangle dz 
\\ & = 
 \bigg\langle D^{2}\phi(\tilde{x}) :\kint_{E_{QA\eps d'_{x_{0}}}^{a}(\tilde{x};\nu)} \hspace{-1em} (z-\tilde{x})\otimes(z-\tilde{x})dz\bigg\rangle
\\  &= q\big(|x_{0}-y_{0}|+Q\eps d'_{x_{0}} \big)^{-q-4}(QA\eps d'_{x_{0}})^{2}\times\\ & \ \bigg\langle  
(q+2)(x_{0}+Q\eps d'_{x_{0}}\nu-y_{0})\otimes
(x_{0}+Q\eps d'_{x_{0}}\nu-y_{0})-(|x_{0}-y_{0}|+Q\eps d'_{x_{0}})^{2}I_{n}
 : V \bigg\rangle ,
\end{align*}
where
$V$ is an $n\times n$ matrix with eigenvalues $a$ for $\nu$ and $1$ for other all eigenvectors.
Thus, we have
\begin{align*}
&\kint_{E_{QA\eps d'_{x_{0}}}^{a}(\tilde{x};\nu)}\langle D^{2}\phi (\tilde{x})(z-\tilde{x}),z-\tilde{x} \rangle dz 
\\ & =  q(aq+a-n+1)\big(|x_{0}-y_{0}|+Q\eps d'_{x_{0}} \big)^{-q-2}(QA\eps d'_{x_{0}})^{2}.
\end{align*}

Similarly, we get
\begin{align*} & \kint_{E_{QA\eps d'_{x_{0}}}^{a}(x_{0}-Q\eps d'_{x_{0}}\nu;\nu)} |z-y_{0}|^{-q}dz
\\ & =
\big(|x_{0}-y_{0}|-Q\eps d'_{x_{0}} \big)^{-q}+ q(aq+a-n+1)\big(|x_{0}-y_{0}|-Q\eps d'_{x_{0}} \big)^{-q-2}(QA\eps d'_{x_{0}})^{2}\\& \qquad+ o(\eps^{2}).
\end{align*}
Hence, we observe that
\begin{align*} &\frac{1}{2}\kint_{E_{QA\eps d'_{x_{0}}}^{a}(x_{0}+Q\eps d'_{x_{0}}\nu;\nu)}\hspace{-5.5em} |z-y_{0}|^{-q}  dz + \frac{1}{2}\kint_{E_{QA\eps d'_{x_{0}}}^{a}(x_{0}-Q\eps d'_{x_{0}}\nu;\nu)} \hspace{-5.5em}|z-y_{0}|^{-q}dz
\\ & =\frac{1}{2}\big( \big(|x_{0}-y_{0}|+Q\eps d'_{x_{0}} \big)^{-q}+\big(|x_{0}-y_{0}|-Q\eps d'_{x_{0}} \big)^{-q}\big) 
\\ & \quad +\frac{ q(aq+a-n+1)(QA\eps d'_{x_{0}})^{2}}{2}\\ & \qquad \qquad \times \big( \big(|x_{0}-y_{0}|+Q\eps d'_{x_{0}} \big)^{-q-2}+\big(|x_{0}-y_{0}|-Q\eps d'_{x_{0}} \big)^{-q-2}\big).
\end{align*}
Note that we have $aq+a-n+1>0 $ by the assumption $q>\frac{n-1}{a}-1 $.
Since the functions $\varphi_{1}(t)=t^{-q}$ and $\varphi_{2}(t)=t^{-q-2}$ are convex,
we obtain that 
\begin{align*} &\frac{1}{2}\kint_{E_{QA\eps d'_{x_{0}}}^{a}(x_{0}+Q\eps d'_{x_{0}}\nu;\nu)}\hspace{-5.5em} |z-y_{0}|^{-q}  dz + \frac{1}{2}\kint_{E_{QA\eps d'_{x_{0}}}^{a}(x_{0}-Q\eps d'_{x_{0}}\nu;\nu)}\hspace{-5.5em} |z-y_{0}|^{-q}dz
\\ & \ge |x_{0}-y_{0}|^{-q}+o(\eps^{2}).
\end{align*}
Combining this with \eqref{fote2}, we get the desired result.
\end{proof}
The above lemma implies a corresponding result to \cite[Lemma 4.6]{han2024tug}.

\begin{lemma}\label{esttaus1inf}Let $\overline{r}<r $ with $r$ as above and $q> \frac{n-1}{a}-1 $ be fixed.
Assume that $|x_{0}-Z^{\rho}(x_{0})  |>\rho-h$ for $h \in (0, \frac{\overline{r}}{2}-\eps)$ and we fix the strategy $S_{II}^{\ast}$ to pull toward $Z_{\rho}(X_{k})$ for Player II.
Then for every small $\eps>0$, $x_{0}\in \overline{\Omega}$, we have
\begin{align}
\sup_{S_{I}}\mathbb{E}_{S_{I},S_{II}^{\ast}}^{x_{0}}[\eps^{2}\overline{\tau}^{\eps,\rho,h,x_{0}}+S_{\overline{\tau}^{\eps,\rho,h,x_{0}}}^{\eps , x_{0}}] \le C\overline{r}^{-q-1}(h+\eps)
\end{align}
for some constant $C$ depending on $n, \alpha, q $ and $\rho$.
\end{lemma}
\begin{proof}

Recall the definition \eqref{deftaus}: $\overline{\tau}^{\eps,\rho,h,x_{0}}=\min\{ k\ge 0 : |Z^{\rho}(X_{k})-X_{k}|<\rho- h \} $ and  $S_{k}^{\eps , x_{0}}= \sum_{j=1}^{k}s_{\eps}(X_{k}) $. 
For simplicity, we write $ \overline{\tau}^{\eps,\rho,h,x_{0}}$ by $ \overline{\tau}$.
Let $q>\max\{n-2, \frac{n-1}{a}-1 \}$ be fixed and $C_{0}>0$ be the constant in Lemma \ref{mcs1infext}.
We set 
$$ Q_{k}= |X_{k}-Z^{\rho}(X_{k})|^{-q}-C_{0}(k\eps^{2}+S_{k}^{\eps,x_{0}}) $$
for $k=0,1,\dots$.

Next, we show thaht $\{ Q_{k} \}_{k=0}^{\overline{\tau}}$ is a submartingale.
First observe that
\begin{align*}
&\sup_{S_{I}}\mathbb{E}_{S_{I},S_{II}^{\ast}}^{x_{0}}[ Q_{k+1} - Q_{k} | \mathcal{F}_{k}]
\\ & =\sup_{S_{I}}\mathbb{E}_{S_{I},S_{II}^{\ast}}^{x_{0}}[ |X_{k+1}-Z^{\rho}(X_{k+1})|^{-n} | \mathcal{F}_{k}] - |X_{k}-Z^{\rho}(X_{k})|^{-n}- C_{0}(\eps^{2}+s_{\eps}(X_{k})) 
\end{align*}
 for each $k < \overline{\tau}$.
From the interior ball condition of $\Omega$, 
we also see that
$$ \bigg|X_{k}-\eps d'_{X_{k}}\frac{X_{k}-Z^{\rho}(X_{k})}{|X_{k}-Z^{\rho}(X_{k})|}\bigg|=|X_{k}-Z^{\rho}(X_{k})|-\eps d'_{X_{k}}$$
and 
$$ \sup_{\nu\in B_{1}}\big|X_{k}+\eps d'_{X_{k}}\nu - Z^{\rho}\big(X_{k}+d'_{X_{k}}\nu \big)\big|=|X_{k}-Z^{\rho}(X_{k})|+\eps d'_{X_{k}}$$
(we have used $h<\frac{\overline{r}}{2}-\eps<\rho$).
On the other hand, it can also be checked that for each $y \in B_{\eps}(X_{k})\cap \Omega$,
$$  |y-Z^{\rho}(X_{k})| \ge |y-Z^{\rho}(y)| = \rho - \dist (y, \partial \Omega).$$
Therefore,  for any strategy $S_{I}$,
\begin{align*}
&\mathbb{E}_{S_{I},S_{II}^{\ast}}^{x_{0}}[ |X_{k+1}-Z^{\rho}(X_{k+1})|^{-n} | \mathcal{F}_{k}]
\\ & \ge \frac{\alpha}{2}\kint_{E_{QA\eps d'_{X_{k}}}^{a}(X_{k}+Q\eps d'_{X_{k}}\nu;\nu)}\hspace{-8em}\big(|z-Z^{\rho}(X_{k})|-\eps d'_{X_{k}}\big)^{-q} dz + \frac{\alpha}{2}\kint_{E_{ QA\eps d'_{X_{k}}}^{a}(X_{k}+Q\eps d'_{X_{k}}\nu;\nu)}\hspace{-8em}\big(|z-Z^{\rho}(X_{k})|+\eps d'_{X_{k}}\big)^{-q} dz 
\\ & \qquad + (1-\alpha)
\kint_{B_{\eps}(X_{k})\cap \Omega}|z-Z^{\rho}(X_{k})|^{-n}dz 
\\ & \ge  |X_{k}-Z^{\rho}(X_{k})|^{-n}+C_{0}(\eps^{2}+s_{\eps}(X_{k})) ,
\end{align*}
where $\nu =\frac{X_{k}-Z^{\rho}(X_{k})}{|X_{k}-Z^{\rho}(X_{k})|} $
by Lemma \ref{mcs1infext}.
This implies that $Q_{k}$ is a supermartingale. 

Since $|Q_{k}|<(\overline{r}-h-\eps)^{-n}$, we can apply the optional stopping theorem.
Hence, for fixed $k\ge 1$, we have
\begin{align*}
\overline{r}^{-n}\le 
\mathbb{E}_{S_{I},S_{II}^{\ast}}^{x_{0}}[Q_{0}] &
\le \mathbb{E}_{S_{I},S_{II}^{\ast}}^{x_{0}}[Q_{\overline{\tau} \wedge k}] 
\\ & \le (\overline{r}-h-\eps)^{-n} - C_{0} \mathbb{E}_{S_{I},S_{II}^{\ast}}^{x_{0}}[(\overline{\tau} \wedge k)\eps^{2}+S_{\overline{\tau} \wedge k}^{\eps,x_{0}}] 
\end{align*}
and this implies
\begin{align*}
\mathbb{E}_{S_{I},S_{II}^{\ast}}^{x_{0}}[\overline{\tau}\eps^{2}+S_{\overline{\tau} }^{\eps,x_{0}}] \le  C\overline{r}^{-n-1}(h+\eps)
\end{align*}
for every strategy $S_{I}$ and some $C=C(n,\rho ,\alpha)$.
The proof is completed by taking the supremum over $S_{I}$.
\end{proof}

Here we state our main theorem which gives boundary H\"{o}lder regularity for functions satisfying the DPP \eqref{dppext} (cf. \cite[Theorem 8.1]{MR4684385} and \cite[Theorem 4.4]{han2024tug}).

\begin{theorem}\label{bdryregext}
Let $u_{\eps}$ be the function satisfying \eqref{dppext}. 
Then there exists $\delta_{0}\in(0,1)$ such that for every $\delta \in (0,\delta_{0})$, $x_{0},y_{0} \in \overline{\Omega}$ 
with $|x_{0}-y_{0}|\le \delta$ and $$\dist (x_{0},\partial \Omega),\dist (y_{0},\partial \Omega)\le \delta^{1/2},$$ 
and some $\sigma \in (0,1)$ in Lemma \ref{intlipext}, we have
\begin{align*}
|u_{\eps}(x_{0})-u_{\eps}(y_{0})|\le C   ||G||_{L^{\infty}(\Gamma_{\eps})}\delta^{\sigma/2} 
\end{align*}
with some uniform $C>0$ depending on $n, \alpha, \gamma, \sigma$ and $\Omega$ and $\eps << \delta$.
\end{theorem}
\begin{proof}

We prove the theorem in three steps. Observe the following decomposition
\begin{align}\label{esdecomp}\begin{split}
&|u_{\eps}(x_{0})-u_{\eps}(y_{0}) |
\\ & \le \big|\sup_{S_{I}}\inf_{S_{II}}\mathbb{E}_{S_{I},S_{II}}^{x_{0}}[u_{\eps}(X_{\overline{\tau}})]-u_{\eps}(x_{0})\big| +\big|\sup_{S_{I}}\inf_{S_{II}}\mathbb{E}_{S_{I},S_{II}}^{y_{0}}[u_{\eps}(Y_{\overline{\tau}})]-u_{\eps}(y_{0})\big|
\\ & \quad +\big| \sup_{S_{I}}\inf_{S_{II}}\mathbb{E}_{S_{I},S_{II}}^{x_{0}}[u_{\eps}(X_{\overline{\tau}})] -\sup_{S_{I}}\inf_{S_{II}}\mathbb{E}_{S_{I},S_{II}}^{y_{0}}[u_{\eps}(Y_{\overline{\tau}})] \big|.\end{split}
\end{align} 
In the first step, we estimate the first two terms of the right hand side in 
The second and third steps are devoted to giving an estimate of the last term.
We construct a supermartingale in the second step and obtain the desired estimate by using it in the last step.

We set $\eps>0$ small enough and $\overline{r}<r<\frac{\rho}{2}$.
Define  $$\{ Z^{\eps, \overline{r}, x_{0}}(X_{k}),Z^{\eps, \overline{r}, y_{0}}(Y_{k}),S_{k}^{\eps, x_{0}},S_{k}^{\eps, y_{0}} \}_{k=0}$$ with $\overline{\tau}^{\eps,\rho,h,x_{0}},\overline{\tau}^{\eps,\rho,h,y_{0}}$, respectively.
We also used the simplified notations $Z_{k}^{\eps, \overline{r}, x_{0}}, Z_{k}^{\eps, \overline{r}, y_{0}}$ and 
$\overline{\tau}.$

We fix $h=\delta^{\sigma/2}$ and consider the following sequence $\{ M_{k} \}_{k=0}$ given by 
\begin{align} \label{mgm} \begin{split} M_{k}&=u_{\eps}(X_{k})\prod_{i=1}^{k}\big( 1-(1-2d'_{X_{i}}) \gamma s_{\eps}(X_{i-1}) \big) \\ & \qquad \quad +\gamma 
\sum_{j=1}^{k-1}\bigg( (1-2d'_{X_{i}}) s_{\eps}(X_{j})G(X_{j})\prod_{i=1}^{j}\big( 1- \gamma s_{\eps}(X_{j-1})) \bigg) .
\end{split}
\end{align} 
Then we can observe that $M_{k}$ is a martingale with respect to $\{ \mathcal{F}_{k} \}_{k=0}$.
Since $M_{k\wedge \overline{\tau}}$
is equibounded by the definition of the value function, we can apply the optional stopping theorem. 
Thus, we have
\begin{align*}
u_{\eps}(x_{0})&=\sup_{S_{I}}\inf_{S_{II}}\mathbb{E}_{S_{I},S_{II}}^{x_{0}}\bigg[u_{\eps}(X_{\overline{\tau}})\prod_{i=1}^{\overline{\tau}}\big(  1-(1-2d'_{X_{i}}) \gamma s_{\eps}(X_{i-1})\big)\\ & \qquad \qquad \qquad \qquad  +\gamma 
\sum_{j=1}^{\overline{\tau}-1}\bigg( (1-2d'_{X_{i}}) s_{\eps}(X_{j})G(X_{j})\prod_{i=1}^{j}\big( 1- \gamma s_{\eps}(X_{j-1}) \bigg)\bigg].
\end{align*}
From the observation 
$$ \prod_{i=1}^{k}(1-a_{i}) \ge 1-\sum_{i=1}^{k}a_{i} $$
for every $k$ and $\{a_{i}\}_{i=1}^{k} \subset [0,1]$,
we get
$$|u_{\eps}(x_{0})-\sup_{S_{I}}\inf_{S_{II}}\mathbb{E}_{S_{I},S_{II}}^{x_{0}}[u_{\eps}(X_{\overline{\tau}})]| \le 2\gamma ||G||_{L^{\infty}(\Gamma_{\eps})} \overline{r}^{-q-1}C\delta^{\sigma/2}  ,$$
$$|u_{\eps}(y_{0})-\sup_{S_{I}}\inf_{S_{II}}\mathbb{E}_{S_{I},S_{II}}^{y_{0}}[u_{\eps}(Y_{\overline{\tau}})]| \le 2\gamma ||G||_{L^{\infty}(\Gamma_{\eps})} \overline{r}^{-q-1}C\delta^{\sigma/2}  ,$$
where $q>\max\{n-2, \frac{n-1}{a}-1 \}$.

Now for our desired result, we just need to get an estimate for $$\big| \sup_{S_{I}}\inf_{S_{II}}\mathbb{E}_{S_{I},S_{II}}^{x_{0}}[u_{\eps}(X_{\overline{\tau}})] -\sup_{S_{I}}\inf_{S_{II}}\mathbb{E}_{S_{I},S_{II}}^{y_{0}}[u_{\eps}(Y_{\overline{\tau}})] \big| .$$ 
Let the strategies for each player with the starting point at $x_{0}$ as $S_{I}^{x},S_{II}^{x}$.
We also write $S_{I}^{y},S_{II}^{y}$ for the strategies with starting point $y_{0}$.
Then we observe that
\begin{align}\label{foreas}\begin{split}
&\big| \sup_{S_{I}}\inf_{S_{II}}\mathbb{E}_{S_{I},S_{II}}^{x_{0}}[u_{\eps}(X_{\overline{\tau}})] -\sup_{S_{I}}\inf_{S_{II}}\mathbb{E}_{S_{I},S_{II}}^{y_{0}}[u_{\eps}(Y_{\overline{\tau}})] \big|
\\ & \le \sup_{S_{I}^{x}}\inf_{S_{II}^{x}}\inf_{S_{I}^{y}}\sup_{S_{II}^{y}}
\mathbb{E}_{S_{I}^{x},S_{II}^{x},S_{I}^{y},S_{II}^{y}}^{(x_{0},y_{0})}[|u_{\eps}(X_{\overline{\tau}})-u_{\eps}(Y_{\overline{\tau}})|] .
\end{split}\end{align}
and
\begin{align} \label{decomp}
\begin{split}
&\mathbb{E}_{S_{I}^{x},S_{II}^{x},S_{I}^{y},S_{II}^{y}}^{(x_{0},y_{0})}[|u_{\eps}(X_{\overline{\tau}})-u_{\eps}(Y_{\overline{\tau}})|] 
\\& = \mathbb{E}_{S_{I}^{x},S_{II}^{x},S_{I}^{y},S_{II}^{y}}^{(x_{0},y_{0})}[|u_{\eps}(X_{\overline{\tau}})-u_{\eps}(Y_{\overline{\tau}})|\mathds{1}_{\{S_{\overline{\tau}}^{\eps , x_{0}}+S_{\overline{\tau}}^{\eps , y_{0}}>1 \}}]
\\ & \ + \mathbb{E}_{S_{I}^{x},S_{II}^{x},S_{I}^{y},S_{II}^{y}}^{(x_{0},y_{0})}[|u_{\eps}(X_{\overline{\tau}})-u_{\eps}(Y_{\overline{\tau}})|\mathds{1}_{\{S_{\overline{\tau}}^{\eps , x_{0}}+S_{\overline{\tau}}^{\eps , y_{0}}\le1 \}}\mathds{1}_{\{|X_{\overline{\tau}}-Y_{\overline{\tau}}|\ge \frac{\delta_{0}}{2} \}}]
\\ & \ + \mathbb{E}_{S_{I}^{x},S_{II}^{x},S_{I}^{y},S_{II}^{y}}^{(x_{0},y_{0})}[|u_{\eps}(X_{\overline{\tau}})-u_{\eps}(Y_{\overline{\tau}})|\mathds{1}_{\{S_{\overline{\tau}}^{\eps , x_{0}}+S_{\overline{\tau}}^{\eps , y_{0}}\le1 \}}\mathds{1}_{\{|X_{\overline{\tau}}-Y_{\overline{\tau}}|< \frac{\delta_{0}}{2}  \}}].
\end{split}
\end{align}
We estimate the term $$ \sup_{S_{I}^{x}}\inf_{S_{II}^{x}}\inf_{S_{I}^{y}}\sup_{S_{II}^{y}}\mathbb{E}_{S_{I}^{x},S_{II}^{x},S_{I}^{y},S_{II}^{y}}^{(x_{0},y_{0})}[|X_{\overline{\tau}}-Y_{\overline{\tau}}|].$$
To do this, we first see
\begin{align*} \sup_{S_{I}^{x}}\inf_{S_{II}^{x}}\inf_{S_{I}^{y}}\sup_{S_{II}^{y}}
\mathbb{E}_{S_{I}^{x},S_{II}^{x},S_{I}^{y},S_{II}^{y}}^{(x_{0},y_{0})}\big[|X_{k+1}-Y_{k+1}|^{2}\big|\mathcal{F}_{k}\big]
\le |X_{k}-Y_{k}|^{2}
\end{align*}
since we can do cancellations for each strategy
when $\dist(X_{k},\partial \Omega),\dist(Y_{k},\partial \Omega)\ge\frac{\eps}{2}$.

Next we consider the other case, that is, $\min\{\dist(X_{k},\partial \Omega),\dist(Y_{k},\partial \Omega)\}<\frac{\eps}{2}$. 
Assume that Player I won the coin toss and let $\nu_{x}\in S^{n-1}$ for the strategy of $X_{k}$. 
For simplicity, we use the notations $X_{k,\nu_{x}}=X_{k}+Q\eps d'_{X_{k}}\nu_{x}$ and 
$Y_{k,\nu_{y}}=Y_{k}+Q\eps d'_{Y_{k}}\nu_{y} $.
When $|X_{k}-Y_{k}|\ge Q\eps (d'_{X_{k}}+d'_{Y_{k}})$,
we take $\nu_{y}$ for the strategy of $Y_{k}$ such that 
$$\nu_{y}=\frac{X_{k,\nu_{x}}-Y_{k}}{|X_{k,\nu_{x}}-Y_{k}|}.$$
For each sample path $\omega$,  
we observe that
\begin{align*}&\int_{B_{1}}\big|(X_{k,\nu_{x}}+QA\eps d'_{X_{k-1}}\mathcal{P}_{\nu_{x}}\mathcal{D}\omega )-(Y_{k,\nu_{y}}+QA\eps d'_{Y_{k-1}}\mathcal{P}_{\nu_{y}}\mathcal{D}\omega )\big|^{2}  d\mathbb{P}(\omega)
\\&
=\int_{B_{1}}\big|(X_{k,\nu_{x}}-QA\eps d'_{X_{k-1}}\mathcal{P}_{\nu_{x}}\mathcal{D}\omega )-(Y_{k,\nu_{y}}-QA\eps d'_{Y_{k-1}}\mathcal{P}_{\nu_{y}}\mathcal{D}\omega )\big|^{2}  d\mathbb{P}(\omega),
\end{align*}
where $\mathbb{P}$ is the probability measure associated with the uniform distribution in $B_{1}$,  
$\mathcal{P}_{\nu}$ is a rotation satisfying $\mathcal{P}_{\nu}\mathbf{e}_{1}=\nu$ for each $\nu\in S^{n-1}$ and
\eqref{conrot}
and $\mathcal{D}$ is a a linear transformation with $\mathcal{D}\mathbf{e}_{1}=a\mathbf{e}_{1} $ and $ \mathcal{D}\mathbf{e}_{j}=\mathbf{e}_{j} $ for $ j=2,\dots,n.$  
This yields
\begin{align*}&\int_{B_{1}}\big|(X_{k,\nu_{x}}+QA\eps d'_{X_{k-1}}\mathcal{P}_{\nu_{x}}\mathcal{D}\omega )-(Y_{k,\nu_{y}}+QA\eps d'_{Y_{k-1}}\mathcal{P}_{\nu_{y}}\mathcal{D}\omega )\big|^{2}  d\mathbb{P}(\omega)
\\&
=\frac{1}{2}\int_{B_{1}}\big(\big|(X_{k,\nu_{x}}+QA\eps d'_{X_{k-1}}\mathcal{P}_{\nu_{x}}\mathcal{D}\omega )-(Y_{k,\nu_{y}}+QA\eps d'_{Y_{k-1}}\mathcal{P}_{\nu_{y}}\mathcal{D}\omega )\big|^{2}  \\ &\qquad \qquad + \big|(X_{k,\nu_{x}}-QA\eps d'_{X_{k-1}}\mathcal{P}_{\nu_{x}}\mathcal{D}\omega )-(Y_{k,\nu_{y}}-QA\eps d'_{Y_{k-1}}\mathcal{P}_{\nu_{y}}\mathcal{D}\omega )\big|^{2}\big) \mathbb{P}(\omega).
\end{align*}
Since
\begin{align*}&\frac{1}{2}\big(\big|(X_{k,\nu_{x}}+QA\eps d'_{X_{k-1}}\mathcal{P}_{\nu_{x}}\mathcal{D}\omega )-(Y_{k,\nu_{y}}+QA\eps d'_{Y_{k-1}}\mathcal{P}_{\nu_{y}}\mathcal{D}\omega )\big|^{2}  \\ &\qquad  + \big|(X_{k,\nu_{x}}-QA\eps d'_{X_{k-1}}\mathcal{P}_{\nu_{x}}\mathcal{D}\omega )-(Y_{k,\nu_{y}}-QA\eps d'_{Y_{k-1}}\mathcal{P}_{\nu_{y}}\mathcal{D}\omega )\big|^{2}\big) 
\\& =|X_{k,\nu_{x}}-Y_{k,\nu_{y}}|^{2}+(QA\eps)^{2}| d'_{X_{k-1}}\mathcal{P}_{\nu_{x}}\mathcal{D}\omega -d'_{Y_{k-1}}\mathcal{P}_{\nu_{y}}\mathcal{D}\omega|^{2}
\\ & \le |X_{k,\nu_{x}}-Y_{k,\nu_{y}}|^{2}+(QA\eps)^{2}|\omega|^{2},
\end{align*}
we get 
\begin{align*}
&\int_{B_{1}}\big|(X_{k,\nu_{x}}+QA\eps d'_{X_{k-1}}\mathcal{P}_{\nu_{x}}\mathcal{D}\omega )-(Y_{k,\nu_{y}}+QA\eps d'_{Y_{k-1}}\mathcal{P}_{\nu_{y}}\mathcal{D}\omega )\big|^{2}  d\mathbb{P}(\omega)\\ & \le  |X_{k,\nu_{x}}-Y_{k,\nu_{y}}|^{2}+(QA\eps)^{2}.
\end{align*}
Thus, we have 
$$\sup_{S_{I}^{x}}\inf_{S_{I}^{y}}
\mathbb{E}_{S_{I}^{x},S_{I}^{y}}^{(x_{0},y_{0})}\big[|X_{k+1}-Y_{k+1}|^{2}\big|\mathcal{F}_{k}\big]\le (|X_{k}-Y_{k}|+Q\eps(d'_{X_{k}}-d'_{Y_{k}}))^{2}+(QA\eps)^{2}.
$$
We similarly derive
$$
\inf_{S_{II}^{x}}\sup_{S_{II}^{y}}\mathbb{E}_{S_{II}^{x},S_{II}^{y}}^{(x_{0},y_{0})}\big[|X_{k+1}-Y_{k+1}|^{2}\big|\mathcal{F}_{k}\big]\le \big(|X_{k}-Y_{k}|+Q\eps(d'_{Y_{k}}-d'_{X_{k}})\big)^{2}+(QA\eps)^{2}.
$$
Meanwhile, for the random walk, one can find the following estimate
\begin{align*}
&\mathbb{E}[|X_{k+1}-Y_{k+1}|^{2}\mathds{1}_{ \{\frac{\alpha}{2}  < \xi_{k} <  1-\frac{\alpha}{2}\}}|\mathcal{F}_{k}\}]
\\&\le (1-\alpha) \big(|X_{k}-Y_{k}|^{2}\big(1+ C(s_{\eps}(X_{k})+s_{\eps}(Y_{k}))\big)+C\eps(s_{\eps}(X_{k})+s_{\eps}(Y_{k}))\big)
\end{align*} for some $C>0$ in the proof of \cite[Theorem 8.1]{MR4684385}.
Now we have
\begin{align*}
& \sup_{S_{I}^{x}}\inf_{S_{II}^{x}}\inf_{S_{I}^{y}}\sup_{S_{II}^{y}}
\mathbb{E}_{S_{I}^{x},S_{II}^{x},S_{I}^{y},S_{II}^{y}}^{(x_{0},y_{0})}\big[|X_{k+1}-Y_{k+1}|^{2}\big|\mathcal{F}_{k}\big]
\\ &\le \frac{\alpha}{2}\big(  |X_{k}-Y_{k}|+\eps (d'_{X_{k}}-d'_{Y_{k}})\big)^{2}
+\frac{\alpha}{2}\big(  |X_{k}-Y_{k}|+\eps (d'_{Y_{k}}-d'_{X_{k}})\big)^{2}+\alpha (QA\eps)^{2}
\\ & \qquad + (1-\alpha) \big( |X_{k}-Y_{k}|^{2}\big(1+ C(s_{\eps}(X_{k})+s_{\eps}(Y_{k}))\big)+C\eps(s_{\eps}(X_{k})+s_{\eps}(Y_{k})))\big) 
\\ & \le |X_{k}-Y_{k}|^{2}\big(1+C(1-\alpha)(s_{\eps}(X_{k})+s_{\eps}(Y_{k}))\big)\\ &\qquad +
(1-\alpha) C\eps(s_{\eps}(X_{k})+s_{\eps}(Y_{k})))+\eps^{2} (d'_{X_{k}}-d'_{Y_{k}})^{2}+\alpha (QA\eps)^{2}.
\end{align*}
We also check that $|d'_{X_{k}}-d'_{Y_{k}}|\le \frac{1}{2}$, and this implies 
$$\eps^{2} (d'_{X_{k}}-d'_{Y_{k}})^{2}+\alpha (QA\eps)^{2} \le C_{0}\eps (s_{\eps}(X_{k})+s_{\eps}(Y_{k})) $$
with some universal $C_{0}>0 $.
Thus, we can obtain 
\begin{align}\label{expbd} \begin{split}
& \sup_{S_{I}^{x}}\inf_{S_{II}^{x}}\inf_{S_{I}^{y}}\sup_{S_{II}^{y}}
\mathbb{E}_{S_{I}^{x},S_{II}^{x},S_{I}^{y},S_{II}^{y}}^{(x_{0},y_{0})}\big[|X_{k+1}-Y_{k+1}|^{2}\big|\mathcal{F}_{k}\big]
\\ & \le |X_{k}-Y_{k}|^{2}\big(1+C(s_{\eps}(X_{k})+s_{\eps}(Y_{k}))\big)+ C\eps(s_{\eps}(X_{k})+s_{\eps}(Y_{k}))).\end{split}
\end{align}
On the other hand, if $|X_{k}-Y_{k}|<Q\eps (d'_{X_{k}}+d'_{Y_{k}})$, we see
$$|X_{k+1}-Y_{k+1}|\le (2+A)Q\eps (d'_{X_{k}}+d'_{Y_{k}}) \le (2+A)Q\eps. $$ 
Thus, we have
\begin{align*}
& \sup_{S_{I}^{x}}\inf_{S_{II}^{x}}\inf_{S_{I}^{y}}\sup_{S_{II}^{y}}
\mathbb{E}_{S_{I}^{x},S_{II}^{x},S_{I}^{y},S_{II}^{y}}^{(x_{0},y_{0})}\big[|X_{k+1}-Y_{k+1}|^{2}\big|\mathcal{F}_{k}\big]
\\ &\le \alpha (Q(2+A)\eps)^{2}
\\ & \qquad + (1-\alpha)  |X_{k}-Y_{k}|^{2}\big(1+ C(s_{\eps}(X_{k})+s_{\eps}(Y_{k}))\big)+C\eps(s_{\eps}(X_{k})+s_{\eps}(Y_{k})))\big) 
\\ & \le 
 C\eps(s_{\eps}(X_{k})+s_{\eps}(Y_{k})))+(1-\alpha)\eps^{2} +\alpha (QA\eps)^{2}
 \\ & \le   C\eps(s_{\eps}(X_{k})+s_{\eps}(Y_{k})))+(1-\alpha)\eps^{2}
\end{align*}
by choosing $C>0$ large enough. 
We have also used that $\eps^{2} \le C_{0}\eps (s_{\eps}(X_{k})+s_{\eps}(Y_{k})) $.

We can construct a supermartingale given by
$$ Q_{k}=|X_{k}-Y_{k}|^{2} e^{-C(S_{k}^{x_{0}}+S_{k}^{y_{0}})}-C\eps (S_{k}^{x_{0}}+S_{k}^{y_{0}}).$$
By employing the optional stopping theorem, we obtain
\begin{align*}
\delta^{2}\ge \sup_{S_{I}^{x}}\inf_{S_{II}^{x}}\inf_{S_{I}^{y}}\sup_{S_{II}^{y}}\mathbb{E}_{S_{I}^{x},S_{II}^{x},S_{I}^{y},S_{II}^{y}}^{(x_{0},y_{0})}[|X_{\overline{\tau}}-Y_{\overline{\tau}}|^{2} e^{-C(S_{\overline{\tau}}^{\eps , x_{0}}+S_{\overline{\tau}}^{\eps , y_{0}})}-C\eps (S_{\overline{\tau}}^{\eps , x_{0}}+S_{\overline{\tau}}^{\eps , y_{0}})].
\end{align*}
We remark that Lemma \ref{esttaus1inf} implies
\begin{align*}
\sup_{S_{I}^{x}}\sup_{S_{I}^{y}}\mathbb{E}_{S_{I}^{x},S_{II}^{x,\ast},S_{I}^{y},S_{II}^{y,\ast}}^{(x_{0},y_{0})}[S_{\overline{\tau}}^{\eps , x_{0}}+S_{\overline{\tau}}^{\eps , y_{0}}]\le C\delta^{1/2},
\end{align*}
and thus
\begin{align}\label{proes}\mathbb{P}(S_{\overline{\tau}}^{\eps , x_{0}}+S_{\overline{\tau}}^{\eps , y_{0}}> 1)\le C\delta^{1/2}. \end{align}
Then we observe the following (cf. \cite[Theorem 4.4]{han2024tug}): 
\begin{align*}
\sup_{S_{I}^{x}}\inf_{S_{II}^{x}}\inf_{S_{I}^{y}}\sup_{S_{II}^{y}}\mathbb{E}_{S_{I}^{x},S_{II}^{x},S_{I}^{y},S_{II}^{y}}^{(x_{0},y_{0})}[|X_{\overline{\tau}}-Y_{\overline{\tau}}|^{2}\mathds{1}_{\{S_{\overline{\tau}}^{\eps , x_{0}}+S_{\overline{\tau}}^{\eps , y_{0}}\le1 \}}]
\le C\delta^{2\sigma}
\end{align*}
and this yields
\begin{align} \label{dtsmes1inf}
\sup_{S_{I}^{x}}\inf_{S_{II}^{x}}\inf_{S_{I}^{y}}\sup_{S_{II}^{y}}\mathbb{E}_{S_{I}^{x},S_{II}^{x},S_{I}^{y},S_{II}^{y}}^{(x_{0},y_{0})}[|X_{\overline{\tau}}-Y_{\overline{\tau}}|^{\sigma}\mathds{1}_{\{S_{\overline{\tau}}^{\eps , x_{0}}+S_{\overline{\tau}}^{\eps , y_{0}}\le1 \}}]\le C\delta^{\sigma} ,
\end{align}
provided $\eps << \delta$.
The following estimate
\begin{align}
\mathbb{P}\bigg( \{S_{\overline{\tau}}^{\eps , x_{0}}+S_{\overline{\tau}}^{\eps , y_{0}}\le 1 \}\cap
\bigg\{ |X_{\overline{\tau}}-Y_{\overline{\tau}}| \ge \frac{\delta_{0}}{2}\bigg\} \bigg) \le  \frac{C\delta}{\delta_{0}}
\end{align}
is also obtained by combining \eqref{proes} with \eqref{dtsmes1inf}.

Then we can obtain
\begin{align*}
&\sup_{S_{I}^{x}}\inf_{S_{II}^{x}}\inf_{S_{I}^{y}}\sup_{S_{II}^{y}}\mathbb{E}_{S_{I}^{x},S_{II}^{x},S_{I}^{y},S_{II}^{y}}^{(x_{0},y_{0})}[|u_{\eps}(X_{\overline{\tau}})-u_{\eps}(Y_{\overline{\tau}})|\mathds{1}_{\{S_{\overline{\tau}}^{\eps , x_{0}}+S_{\overline{\tau}}^{\eps , y_{0}}\le1 \}}\mathds{1}_{\{|X_{\overline{\tau}}-Y_{\overline{\tau}}|< \frac{\delta_{0}}{2}  \}}]
\\ & \le C  ||G||_{L^{\infty}(\Gamma_{\eps})}\delta^{\sigma/2}.
\end{align*}
Finally, we have
\begin{align*}
&\sup_{S_{I}^{x}}\inf_{S_{II}^{x}}\inf_{S_{I}^{y}}\sup_{S_{II}^{y}}\mathbb{E}_{S_{I}^{x},S_{II}^{x},S_{I}^{y},S_{II}^{y}}^{(x_{0},y_{0})}[|u_{\eps}(X_{\overline{\tau}})-u_{\eps}(Y_{\overline{\tau}})|] \le C  ||G||_{L^{\infty}(\Gamma_{\eps})}\delta^{\sigma/2}.
\end{align*}
This completes the proof.
\end{proof}

\subsection{Convergence} 
We consider the convergence issue of value functions as $\eps \to 0$ in this subsection. 
Before proving the convergence, we first introduce the notion of viscosity solutions.

\begin{definition}[viscosity solution] \label{viscosity}
A function $u \in C(\overline{\Omega})$ is a viscosity solution to \eqref{defcapf} if
for all $x \in \overline{\Omega}$ and $\phi \in C^{2}$ such that $u(x)=\phi(x) $ and $u(y)> \phi(y) $ for $y \neq x$, we have
\begin{align*}
\left\{ \begin{array}{ll}
 \Delta_{p}^{N}\phi(x) \le 0 & \textrm{if $ x \in \Omega$,}\\
\min\big\{ \Delta_{p}^{N}\phi(x) , \gamma_{0} G(x)-\big( \langle\mathbf{n} , D\phi\rangle (x)  + \gamma_{0} \phi(x)\big) \big\} \le 0 & \textrm{if $ x \in \partial \Omega$,}\\
\end{array} \right.
\end{align*}
and 
for all $x \in \overline{\Omega}$ and $\phi \in C^{2}$ such that $u(x)=\phi(x) $ and $u(y)< \phi(y) $ for $y \neq x$, we have
\begin{align*}
\left\{ \begin{array}{ll}
 \Delta_{p}^{N}\phi (x)\ge 0 & \textrm{if $ x \in \Omega$,}\\
\max\big\{\Delta_{p}^{N}\phi(x),  \gamma_{0} G(x)-\big( \langle\mathbf{n} , D\phi\rangle(x)  + \gamma_{0} \phi(x)\big)  \big\}\ge 0   & \textrm{if $ x \in \partial \Omega$.}\\
\end{array} \right.
\end{align*}
\end{definition}

Then we can show the convergence of the value function to a viscosity solution of \eqref{defcapf}.
\begin{theorem} \label{convrob12}
Let $u_{\eps}$ be the function satisfying \eqref{dppext} and $G \in C^{1}(\overline{\Gamma}_{\eps})$ for each $\eps >0$.
Then, there exists a continuous function $u: \overline{\Omega} \to \mathbb{R}^{n}$ and
a subsequence $\{  \eps_{i}\}$ such that $u_{\eps_{i}}$ converges uniformly to $u$ on $\overline{\Omega}$ and
$u$ is a viscosity solution to the problem \eqref{defcapf}.
\end{theorem}
\begin{proof}
For the interior case, we can prove the convergence by using \cite[Corollary 7.3]{MR4257616}.
Thus, we only need to consider the boundary case.
First, we fix
$x \in \partial \Omega$ and $\phi \in C^{2}(\overline{\Omega} )$
such that $u-\phi$ attains a strict local minimum at $x$
(we use the notation $\{ u_{\eps} \}$ instead of $\{ u_{\eps_{i}} \}$ for simplicity).
Since $u_{\eps} $ uniformly converges to $u$,  we have
\begin{align}\label{lu1}\inf_{B_{r}(x)}(u_{\eps}-\phi) < u_{\eps}(z)-\phi(z) 
\end{align}
for all $z\in B_{r}(x)\backslash \{ x\}$ for sufficiently small $\eps >0$.
Hence, we can choose a point $x_{\eps} \in B_{r}(x) \cap \overline{\Omega}$ for any $ \zeta_{\eps}>0  $ such that
\begin{align}\label{lu2} u_{\eps}(x_{\eps})-\phi(x_{\eps}) \le  u_{\eps}(z)-\phi(z)+ \zeta_{\eps}  \end{align}
for any $z \in B_{r}(x)$ and small $\eps >0$.
Let $\varphi = \phi +  u_{\eps}(x_{\eps})-\phi(x_{\eps})$.
Since we can observe that $x_{\eps} \to x$ as $\eps \to 0$,
we have $\varphi (x_{\eps})= u_{\eps}(x_{\eps})$ and
\begin{align}\label{lu3}u_{\eps}(z) \ge u_{\eps}(x_{\eps})-\phi(x_{\eps})+\phi(z)-\zeta_{\eps} = \varphi(z)-\zeta_{\eps}
\end{align}
for each $z \in B_{r}(x)$.

Set 
$\tilde{T}_{\eps}^{G}u(x) =\big( 1-\gamma s_{\eps}(x) \big) \mathcal{A}u_{\eps}(x) +
\gamma s_{\eps}(x) G(x)$
with $\gamma = (1-\alpha)\gamma_{0}$.
Since $u_{\eps}$ satisfies $u_{\eps}= \tilde{T}_{\eps}^{G}u_{\eps} $ and
\begin{align*}u_{\eps}(z) \ge  \varphi(z)-\zeta_{\eps}
\end{align*}
for each $z \in B_{r}(x)$, we obtain that
\begin{align} \label{zetaep12}
\zeta_{\eps} \ge  \tilde{T}_{\eps}^{G}\phi (x_{\eps})- \phi (x_{\eps}) +\gamma  s_{\eps}(x_{\eps}) \Lambda_{\eps}
\end{align}
with $ \Lambda_{\eps}=\zeta_{\eps}+\phi (x_{\eps}) -u_{\eps}(x_{\eps}) $.

Now we assume that there is an infinite sequence $\{x_{\eps_{j}}\} $ such that $x_{\eps_{j}} \in \Gamma_{\eps_{j}}$ for each $j$ (for simplicity, we maintain the notation $x_{\eps}$). 
Recall that
\begin{align*}\mathcal{A}\phi(x)=&\frac{\alpha}{2} \bigg\{\sup_{z\in \partial B_{Q\eps d'_{x}}(x)}\kint_{E_{QA\eps d'_{x}}^{a}(z;\nu_{0})} \hspace{-2em}\phi(y)dy+\inf_{z\in \partial B_{Q\eps d'_{x}}(x)} \kint_{E_{QA\eps d'_{x}}^{a}(z;\nu_{0})}\hspace{-2em}\phi(y)dy \bigg\}\\ & \quad + (1- \alpha)\kint_{B_{\eps}(x)\cap \Omega}u_{\eps}(y)dy.
\end{align*}
We first observe that
\begin{align}
\label{tay1}
\begin{split}&\frac{1}{2}\bigg\{\kint_{E_{QA\eps/2}^{a}(x_{\eps}+w;\nu_{0})}\phi(y)dy+ \kint_{E_{QA\eps/2}^{a}(x_{\eps}-w;\nu_{0})} \phi(y)dy \bigg\}
\\& =  \frac{1}{2} \big( \phi(x_{\eps}+w)+\phi(x_{\eps}-w) \big) +O(\eps^{2})=\phi(x_{\eps})+O(\eps^{2})\end{split}
\end{align} for any $w \in \partial B_{Q\eps /2}$
by using Taylor expansion. 
Combining this with
\begin{align*}   \kint_{B_{\eps}(x_{\eps})\cap \Omega}\phi(y)dy  
& =\phi(x_{\eps})+\bigg\langle D\phi(x_{\eps}),  \kint_{B_{\eps}(x_{\eps})\cap \Omega }(y-x_{\eps})dy \bigg\rangle
\\ &\quad + \frac{1}{2}\bigg\langle D^{2}\phi(x_{\eps}):\kint_{B_{\eps}(x_{\eps})\cap \Omega }(y-x_{\eps})\otimes(y-x_{\eps})dy\bigg\rangle + o(\eps^{2})
\\& =    \phi(x_{\eps})  -s_{\eps}(x_{\eps}) \langle D\phi(x_{\eps}), \mathbf{n}(\pi_{\partial \Omega}x_{\eps})\rangle  + O(\eps s_{\eps}(x_{\eps})  ),
\end{align*} 
we get
\begin{align*}\begin{split}
&\big( 1-\gamma s_{\eps}(x_{\eps}) \big) \mathcal{A}u_{\eps}(x_{\eps}) +
\gamma s_{\eps}(x_{\eps}) G(x_{\eps})\\ &=
\big( 1-\gamma s_{\eps}(x_{\eps}) \big)\times 
 \\ & \ \bigg\{ \frac{\alpha}{2}\kint_{E_{QA\eps/2}^{a}(x_{\eps}+w;\nu_{0})}\hspace{-1.5em}\phi(y)dy+ \frac{\alpha}{2}\kint_{E_{QA\eps/2}^{a}(x_{\eps}-w;\nu_{0})} \hspace{-1.5em}\phi(y)dy + (1-\alpha) \kint_{B_{\eps}(x_{\eps})\cap \Omega}\phi(y) dy \bigg\}
 \\& \qquad + \gamma  s_{\eps}(x_{\eps})G(x_{\eps}) 
 \\ & = (1-\gamma s_{\eps}(x_{\eps}))   \big( \phi(x_{\eps})   -s_{\eps}(x_{\eps}) (1-\alpha)\langle D\phi(\pi_{\partial \Omega}x_{\eps}), \mathbf{n}(\pi_{\partial \Omega}x_{\eps})\rangle \big)
 \\ & \qquad +\gamma  s_{\eps}(x_{\eps})G(x_{\eps}) + O(\eps s_{\eps}(x_{\eps})  ).
 \end{split}
\end{align*}

Since $\phi \in C^{2}(\overline{\Omega})$, we can always choose $w_{1} \in \partial B_{Q\eps /2}$ be a unit vector such that
$$ \kint_{E_{QA\eps/2}^{a}(x_{\eps}+w_{1};\nu_{0})} \phi(y)dy = \inf_{z\in \partial B_{Q\eps d'_{x_{\eps}}}(x_{\eps})} \kint_{E_{QA\eps/2}^{a}(z;\nu_{0})} \phi(y)dy  .  $$
This implies
\begin{align*}\begin{split}& \frac{\alpha}{2} \bigg\{\sup_{z\in \partial B_{Q\eps d'_{x_{\eps}}}(x_{\eps})}\kint_{E_{QA\eps d'_{x_{\eps}}}^{a}(z;\nu_{0})} \hspace{-2em}\phi(y)dy+\inf_{z\in \partial B_{Q\eps d'_{x_{\eps}}}(x_{\eps})} \kint_{E_{QA\eps d'_{x_{\eps}}}^{a}(z;\nu_{0})}\hspace{-2em}\phi(y)dy \bigg\}\\ & \quad + (1- \alpha)\kint_{B_{\eps}(x_{\eps})\cap \Omega}u_{\eps}(y)dy
\\ & \ge \frac{\alpha}{2}\kint_{E_{QA\eps/2}^{a}(x_{\eps}+w_{1};\nu_{0})}\hspace{-1.5em}\phi(y)dy+ \frac{\alpha}{2}\kint_{E_{QA\eps/2}^{a}(x_{\eps}-w_{1};\nu_{0})} \hspace{-1.5em}\phi(y)dy + (1-\alpha) \kint_{B_{\eps}(x_{\eps})\cap \Omega}\phi(y) dy .
\end{split}
\end{align*}
Recall that 
$$\tilde{T}_{\eps}^{G}u(x) =\big( 1-\gamma s_{\eps}(x) \big) \mathcal{A}u_{\eps}(x) +
\gamma s_{\eps}(x) G(x).$$
Then we have
\begin{align*} \tilde{T}_{\eps}^{G}\phi (x_{\eps})
& \ge  \phi(x_{\eps})  + \big(\gamma  G(x_{\eps})- (\gamma \phi(x_{\eps})+  (1-\alpha)\langle D\phi(\pi_{\partial \Omega}x_{\eps}), \mathbf{n}(\pi_{\partial \Omega}x_{\eps})\rangle ) \big) s_{\eps} (x_{\eps}) \\ & \qquad + O(\eps s_{\eps}(x_{\eps})  ) .
\end{align*} 
Thus, we obtain \begin{align}\label{bdconde}\begin{split}&  \zeta_{\eps} \ge  \big(\gamma  G(x_{\eps})- (\gamma \phi(x_{\eps})+(1-\alpha)  \langle D\phi(\pi_{\partial \Omega}x_{\eps}), \mathbf{n}(\pi_{\partial \Omega}x_{\eps})\rangle ) \big)s_{\eps} (x_{\eps})\\ & \qquad + \gamma  s_{\eps}(x_{\eps}) \Lambda_{\eps} + O(\eps s_{\eps}(x_{\eps})  ),\end{split}
\end{align} and this yields
\begin{align*}0& \ge  \gamma  G(x)- \gamma  \phi(x)- (1-\alpha)\langle D\phi(x), \mathbf{n}(x)\rangle 
\end{align*}by the same calculation for the $p>2$ case (see the proof of \cite[Theorem 5.2]{han2024tug}).
 Hence we have 
$$ \frac{\partial \phi}{\partial \mathbf{n}}(x)+\gamma_{0}  \phi(x) \ge \gamma_{0} G(x).$$
The reversed inequality can be obtained in much the same way and we get the desired result.
\end{proof}


\end{document}